\documentclass[UTF-8,reqno]{amsart}
\usepackage{enumerate}
\usepackage{amssymb,url,color, booktabs}

\usepackage{mathrsfs}
\usepackage{amsmath}

\usepackage{fancyhdr}
\usepackage[nobysame]{amsrefs}
\BibSpec{article}{%
+{}{\PrintAuthors} {author}
+{,}{ \textrm} {title}
+{.}{ \textit} {journal}
+{,}{ \textbf} {volume}
+{}{ \parenthesize} {date}
+{,}{ } {pages}
+{.}{ arXiv:} {eprint}
+{.}{} {transition}
}
\BibSpec{book}{%
+{}{\PrintAuthors} {author}
+{,}{ \textit} {title}
+{.}{ \textrm} {series} 
+{,}{ Vol.} {volume}
+{.}{ } {publisher}
+{,}{ } {date}
+{.}{} {transition}
}
\usepackage{color}
\usepackage[colorlinks=true]{hyperref}
\hypersetup{
    linkcolor=blue,          
    citecolor=red,        
    filecolor=blue,      
    urlcolor=cyan
}

\numberwithin{equation}{section}

\newcommand{\be}{\begin{eqnarray}}
\newcommand{\ee}{\end{eqnarray}}
\newcommand{\ce}{\begin{eqnarray*}}
\newcommand{\de}{\end{eqnarray*}}

\newenvironment{proof of theorem t3}{{\it Proof of Theorem \ref{t3}}.}{{\hfill 	
$\square$\hskip - \parfillskip}}
\newenvironment{proof of theorem 1.2}{{\it Proof of Theorem 1.2}.}{{\hfill 	
		$\square$\hskip - \parfillskip}}
\newenvironment{proof of theorem 1.3}{{\it Proof of Theorem \ref{t2}}.}{{\hfill 	
		$\square$\hskip - \parfillskip}}
\newenvironment{proof of theorem 4.1}{{\it Proof of Theorem \ref{t4.1}}.}{{\hfill 	
			$\square$\hskip - \parfillskip}}
\newenvironment{proof of (1.13)}{{\it Proof of (\ref{5.9})}.}{{\hfill 	
		$\square$\hskip - \parfillskip}}

\newtheorem{theorem}{Theorem}[section]
\newtheorem{lemma}[theorem]{Lemma}
\newtheorem{conjecture}[theorem]{Conjecture}
\newtheorem{remark}[theorem]{Remark}
\newtheorem{definition}[theorem]{Definition}
\newtheorem{proposition}[theorem]{Proposition}
\newtheorem{Examples}[theorem]{Example}
\newtheorem{corollary}[theorem]{Corollary}

\makeatletter
\newcommand{\rmnum}[1]{\romannumeral #1}
\newcommand{\Rmnum}[1]{\expandafter\@slowromancap\romannumeral #1@}
\makeatother

\def\eps{\varepsilon}

\def\e{\mathrm{e}}

\def\om{\omega}
\def\Om{\Omega}

\def\p{\partial}

\def\g{\gamma}
\def\l{\lambda}
\def\la{\langle}
\def\ra{\rangle}
\def\[{{\Big[}}
\def\]{{\Big]}}
\def\<{{\langle}}
\def\>{{\rangle}}
\def\({{\Big(}}
\def\){{\Big)}}

\def\bx{{\mathbf{x}}}

\def\min{{\mathord{{\rm min}}}}
\def\Vol{\mathord{{\rm Vol}}}

\def\={&\!\!=\!\!&}

\def\cL{{\mathcal L}}

\def\mH{{\mathbb H}}

\def\mK{{\mathbb K}}

\def\mR{{\mathbb R}}
\def\mS{{\mathbb S}}

\def\1{{\mathbf{1}}}

\def\sF{{\mathscr F}}

\def\geq{\geqslant}
\def\leq{\leqslant}
\def\ge{\geqslant}
\def\le{\leqslant}

\def\k{\kappa}

\def\eps{\varepsilon}

\def\e{\mathrm{e}}

\def\om{\omega}
\def\Om{\Omega}

\def\p{\partial}

\def\g{\gamma}
\def\l{\lambda}
\def\la{\langle}
\def\ra{\rangle}
\def\[{{\Big[}}
\def\]{{\Big]}}
\def\<{{\langle}}
\def\>{{\rangle}}
\def\({{\Big(}}
\def\){{\Big)}}

\def\bx{{\mathbf{x}}}

\def\min{{\mathord{{\rm min}}}}
\def\Vol{\mathord{{\rm Vol}}}

\def\={&\!\!=\!\!&}
\def\bt{\begin{theorem}}
\def\et{\end{theorem}}
\def\bl{\begin{lemma}}
\def\el{\end{lemma}}
\def\br{\begin{remark}}
\def\er{\end{remark}}
\def\bx{\begin{Examples}}
\def\ex{\end{Examples}}
\def\bd{\begin{definition}}
\def\ed{\end{definition}}
\def\bp{\begin{proposition}}
\def\ep{\end{proposition}}
\def\bc{\begin{corollary}}
\def\ec{\end{corollary}}

\def\geq{\geqslant}
\def\leq{\leqslant}
\def\ge{\geqslant}
\def\le{\leqslant}

 \def\nn{\nabla}

\def\<{\langle} \def\>{\rangle}

\def\bpf{\begin{proof}}
\def\epf{\end{proof}}

\allowdisplaybreaks

\begin{document}
	
\title{locally constrained flows and geometric inequalities in sphere}\thanks{\it {This research is partially supported by NSFC (Nos. 11871053 and 12261105).}}
\author{Shanwei Ding and Guanghan Li}

\thanks{{\it 2020 Mathematics Subject Classification: 53C44, 35K55.}}
\thanks{{\it Keywords: locally constrained flow, geometric inequalities, weighted curvature
integral, quermassintegral, curvature function}}
\thanks{\it data availability statement: Data sharing not applicable to this article as no datasets were generated or analysed during the current study.}

\address{School of Mathematics and Statistics, Wuhan University, Wuhan 430072, China.
}

\begin{abstract}
In this paper, we uncover an intriguing algebra property of an element symmetric polynomial. By this property, we establish the longtime existence and convergence of a locally constrained flow, thereby some families of geometric inequalities in sphere can be derived. Meanwhile, a new family of ``three terms'' geometric inequalities involving two weighted curvature
integrals and one quermassintegral are proved. Unlike hyperbolic spaces, we also obtain an inverse weighted geometric inequality in sphere.
\end{abstract}

\maketitle
\setcounter{tocdepth}{2}
\tableofcontents

\section{Introduction}
If $M\subset\mR^{n+1}$ is a convex hypersurface, the family of Alexandrov-Fenchel inequalities is the following,
\begin{align*}
\int_{M}p_k(\k)d\mu\ge\vert\mS^n\vert^\frac{1}{n-k+1}\(\int_{M}p_{k-1}(\k)d\mu\)^\frac{n-k}{n-k+1},
\end{align*}
where $p_k(\k)$ is defined as the normalized $k$th elementary symmetric function of the principal curvatures $\k=(\k_1,\cdots,\k_n)$ of $M$, $\vert\mS^n\vert$ denotes the area of $\mS^n$. Classically, these inequalities are raised in the convex body's theory thus convexity is used in an essential way. More details can be seen in \cite{SR}. Afterwards, Guan and Li \cite{GL2} made a pioneering work that they used a suitable inverse curvature flow to derive the same inequalities for any $k$-convex and star-shaped hypersurface. Thus
it is natural to generalize the Alexandrov-Fenchel inequalities to a larger class of smooth bounded domains in space form.

It is known that the space form $\mK^{n+1}$ can be viewed as Euclidean space $\mathbb{R}^{n+1}$ equipped with a metric tensor, i.e., $\mathbb{K}^{n+1}=(\mathbb{R}^{n+1},ds^2)$ with proper choice $ds^2$. More specifically, let $\mathbb S^n$ be the unit sphere in Euclidean space $\mathbb{R}^{n+1}$ with standard induced metric $dz^2$, then
\begin{equation*}
	\bar{g}:=ds^2=d\rho^2+\phi^2(\rho)dz^2,
\end{equation*}
where
\begin{equation*}
	\phi(\rho)=\begin{cases}
\sinh(\rho), &\rho\in[0,\infty), \qquad K=-1,\\
\rho, &\rho\in[0,\infty), \qquad K=0,\\
\sin(\rho), &\rho\in[0,\pi),\qquad K=1.
	\end{cases}
\end{equation*}
We define
\begin{equation*}
	\Phi(\rho)=\int_0^\rho\phi(s)ds=\begin{cases}
		\cosh(\rho)-1, \qquad &K=-1,\\
		\frac{\rho^2}{2},\qquad &K=0,\\
		1-\cos(\rho),\qquad &K=1.
	\end{cases}
\end{equation*}
Consider the vector field $V=\bar\nn\Phi=\phi(\rho)\frac{\p}{\p\rho}$ on $\mathbb{K}^{n+1}$. We know that $V$ is a conformal killing field, i.e.  $V$ satisfies $\bar{\nabla}_YV=\phi'(\rho)Y$ by \cite{GL}. We first review the definition of the $k$th quermassintegral $W_k$ in space form.
For any bounded domain $\Omega$ with smooth boundary $M=\p\Omega$ in the space form $\mathbb{K}^{n+1}$, the $k$th quermassintegral $W_k$ is defined as the measure of the set of totally geodesic $k$-dimensional subspaces which intersect $\Omega$. We have
\begin{equation*}
W_0(\Om)=\Vol(\Om), \quad W_1(\Om)=\Vol(\p\Om), \quad W_{n+1}(\Om)=\frac{\om_n}{n+1},
\end{equation*}
where $\om_n=\Vol(\mS^n)$. The quermassintegrals are related to the curvature integrals of $M$ by (\cite{SG})
\begin{equation}\label{r1}
	\begin{split}
&\int_{M}p_k(\k)d\mu=(n-k)W_{k+1}(\Om)-kKW_{k-1}(\Om), \quad k=1,\cdots,n-1,\\
&\int_{M}p_n(\k)d\mu=\om_n-nKW_{n-1}(\Om).
	\end{split}
\end{equation}
The quermassintegrals satisfy the following nice variational property:
\begin{equation*}
\frac{d}{dt}W_k(\Om_t)=\dfrac{n+1-k}{n+1}\int_{\p\Omega_t}\mathcal{F} p_k(\k)d\mu_t,\qquad 0\le k\le n
\end{equation*}
along any normal variation with speed $\mathcal{F}$. For the convenience of description, we give the following definitions.
\begin{definition}
A smooth bounded domain $\Om$ in space form $\mK^{n+1}$ is called $k$-convex (resp. strictly $k$-convex) if the principal curvatures of the boundary $\p\Om$ satisfy $\k\in\bar{\Gamma}^+_k$ (resp. $\k\in\Gamma^+_k$), where $\Gamma^+_k=\{\k\in\mR^n|p_m(\k)>0,\forall m\le k\}$;

A smooth bounded domain $\Om$ in space form $\mK^{n+1}$ is called convex (resp. strictly convex) if the principal curvatures of the boundary $\p\Om$ satisfy $\k\in\bar{\Gamma}^+_n$ (resp. $\k\in\Gamma^+_n$), i.e. $\k\in\{\forall i,\k_i\ge0\}$ (resp. $\k\in\{\forall i,\k_i>0\}$);

A smooth bounded domain $\Om$ in hyperbolic space $\mH^{n+1}$ is said to be $h$-convex (resp. strictly $h$-convex) if the principal curvatures of the boundary $\p\Om$ satisfy $\k_i\ge1$ (resp. $\k_i>1$) for all $i=1,\cdots,n$.
\end{definition}

As direct extensions of the Alexandrov-Fenchel inequalities in Euclidean space, the following Alexandrov-Fenchel type inequalities for $h$-convex domains in hyperbolic space were proved in
\cite{WX,HLW} ,
\begin{equation}\label{1.1}
W_k(\Om)\ge f_k\circ f_l^{-1}(W_l(\Om)),\quad 0\le l<k\le n.
\end{equation}
Equality holds if and only if $\Om$ is a geodesic ball. Here $f_k:[0,\infty)\rightarrow\mR^+$ is a monotone function defined by $f_k(r)=W_k(B_r)$, the $k$th quermassintegral for the geodesic ball of radius $r$, and $f_l^{-1}$ is the inverse function of $f_l$. 
If $\p\Om$ is not $h$-convex, there are also some results in \cite{LWX,HL,AHL,ACW} etc.. In $\mS^{n+1}$, there are only a few results for bounded convex domains, and we list \cite{BGL,CS,HL3,CGL} as references.
 
In addition, it can be found that the quermassintegrals, the curvature integrals and the weighted curvature integrals (such as $\int\phi'p_kd\mu$) differ only by a constant multiple in $\mR^{n+1}$. However, these three quantities differ greatly in hyperbolic space and sphere. Therefore
in addition to the quermassintegral inequalities, there is great interest to establish the weighted geometric inequalities in hyperbolic space and sphere, which can also be seen as an extension of another version of the Alexandrov-Fenchel inequalities in $\mR^{n+1}$. In \cite{BHW}, Brendle, Hung and Wang first initiated to prove weighted geometric inequalities involving $\int_{M}\phi'p_1d\mu$, $\Vol(M)$ and $\int_{M}ud\mu$ by applying curvature flows in hyperbolic space, and  a similar inequality was obtained in \cite{DG}.
Recently, \cite{HLW} and \cite{HL2} proved the following higher order weighted geometric inequalities for $h$-convex hypersurface and static convex hypersurface respectively in hyperbolic space,
\begin{equation}\label{x1.8}
\int_{M}\phi'p_kd\mu\ge h_k\circ f_l^{-1}(W_l(\Om)), \quad 1\le k\le n, 0\le l\le k.
\end{equation}
Equality holds if and only if $\Om$ is a geodesic ball. Here $h_k:[0,\infty)\rightarrow\mR^+$ is a monotone function defined by $h_k(r)=\om_n\sinh^{n-k}r\cosh^{k+1}r$, $\int_{M}\phi'p_kd\mu$ with $M=\p B_r$, and $h_k^{-1}$ is the inverse function of $h_k$. This raises two interesting questions: The first question is whether more sharp inequalities can be obtained, and the second question is whether two weighted curvature integrals can be compared. We shall give partial answers to these two questions in this paper.

For any bounded domain $\Om$ with smooth boundary $M=\p\Om$, we introduce the weighted curvature integral as follows:
\begin{align*}
	W_{-1}^{\phi'}(\Om)&=\int_{M}ud\mu=\int_\Om(n+1)\phi'd\text{vol},\quad W_{n}^{\phi'}(\Om)=\int_{M}\phi'p_nd\mu,\\
	W_{k-1}^{\phi'}(\Om)&=\int_{M}\phi'p_{k-1}d\mu=\int_{M}up_{k}d\mu,\quad k=1,\cdots,n.
\end{align*}
Along the outward normal variation with speed $\mathcal{F}$, by direct calculation, the weighted curvature integral evolves by (see \cite{HL2} Proposition 3.1)
\begin{equation}\label{1.6}
	\frac{d}{dt}W_{k-1}^{\phi'}(\Om_t)=\int_{M_t}(-Kkup_{k-1}+(n+1-k)\phi'p_k)\mathcal{F}d\mu_t,\quad k=0,\cdots,n+1,
\end{equation}
where we take $p_{-1}=p_{n+1}=0$ by convention. For these two questions, there have some partial answers.
For the second question, Hu and Li in \cite{HL2} derived
$$W_k^{\phi'}(\Om)\ge h_k\circ h_{-1}^{-1}(W_{-1}^{\phi'}(\Om))$$
for static convex hypersurfaces.
For the first question, a family of fascinating quantities $\int_{M}\Phi p_k(\k)d\mu$ in hyperbolic space and sphere have been considered. These quantities in $\mR^{n+1}$ have been studied widely. For instance, \cite{GR,KM,KM2} etc..
We hope to consider the following geometric inequalities.
\begin{equation}\label{x1.10}
\int_{M}\Phi p_kd\mu+kW_{k-1}(\Om)\ge (\xi_k+kf_{k-1})\circ f_l^{-1}(W_l(\Om)), \quad 1\le k\le n, 0\le l\le k.
\end{equation}
In \cite{KW}, Kwong and Wei proved the inequalities for strictly convex hypersurfaces in sphere with $k=n$ and static convex hypersurfaces in hyperbolic space with arbitrary $k$.

We give some motivations for considering (\ref{x1.10}). Firstly, by (\ref{r1}), the left hand side of (\ref{x1.10}) in hyperbolic space  can be written as $\int_{M}\phi' p_kd\mu-(n-k)W_{k+1}(\Om)$. This implies (\ref{x1.10}) is more sharp than the weighted geometric inequalities (\ref{x1.8}). Thus this gives an answer to the second question. Secondly, (\ref{x1.10})  has been extensively studied in $\mR^{n+1}$ and has many important applications. For example, in \cite{BFN}, a similar inequality with $k=0$ is used to prove some Weinstock type inequalities for the first nonzero Steklov and Wenzell eigenvalue for convex domains in $\mR^{n+1}$. Thirdly, the inequalities can be thought as a higher dimensional version of the results in \cite{KW}. In particular, (\ref{x1.10}) in sphere is an analogue of the conjecture in \cite{GP}. Finally, if we shall derive some inequalities by using geometric flows, we could only compare two geometric quantities usually since the way we take is to preserve one quantity while the another is monotone. Therefore this is very difficult if one wants to obtain inequalities involving more than two geometric quantities. This situation can be improved by constructing (\ref{x1.10}).

Firstly we shall derive the following theorem.
\begin{theorem}\label{t4.1}
	Assume $X_0(M)$ is the embedding of a closed, smooth, strictly convex $n$-dimensional manifold $M^n$ in $\mathbb{S}^{n+1}$. Let $F=p_2^\frac{1}{2}(\k_i)$,  then the flow 
	\begin{equation}\label{x1.7}
\frac{\p}{\p t}X=\(\phi'-uF\)\nu,
	\end{equation}
	 has a unique smooth, strictly convex solution $M_t$ for all time $t>0$, and the solution converges exponentially to a geodesic slice in the $C^\infty$-topology.
\end{theorem}
\textbf{Remark:} In \cite{GL}, Guan and Li have studied flow (\ref{x1.7}) with $F=p_1$. Since this is a difficult and unsolved problem if $F=\frac{p_{k+1}}{p_k}$, we give another extension of the result in \cite{GL} by considering $F=p_k^\frac{1}{k}(\k_i)$ as a substitute for $F=\frac{p_{k+1}}{p_k}$, which can also obtain many geometric inequalities. An example that using the curvature function $F=p_k^\frac{1}{k}(\k_i)$ can prove geometric inequalities can be seen in \cite{DL}.

In this paper, by the result of flows we can derive a new family of "three terms" geometric inequalities involving two weighted curvature integrals and one quermassintegral. To the best of our knowledge, there is no inequality involving two weighted integrals in sphere. We obtain such inequalities, which can be seen as an analogue of \cite{BHW} in sphere. We also derive some ``three terms" geometric inequalities involving one weighted curvature integrals and two quermassintegral.
\begin{theorem}\label{t5.7}
	Let $\Om$ be a strictly convex domain with smooth boundary $M$ in $\mS^{n+1}$. If $0\le k\le n$, there hold the following inequalities
	\begin{align}\label{5.9}
\int_{M}\Phi p_kd\mu+kW_{k-1}(\Om)\ge (\xi_k+kf_{k-1})\circ h_{-1}^{-1}(W_{-1}^{\phi'}(\Om)),
	\end{align}
and if $0\le k\le n-2$,
	\begin{align}\label{5.12}
\int_{M}\Phi p_kd\mu+kW_{k-1}(\Om)\ge (\xi_k+kf_{k-1})\circ f_{0}^{-1}(W_{0}(\Om)).
\end{align}
These	equalities hold if and only if $\Om$ is a geodesic ball centered at the origin.
\end{theorem}
\textbf{Remark:} If $k=1$, (\ref{5.9}) becomes
\begin{equation}\label{5.10}
	\int_{M}\Phi p_1d\mu+W_{0}(\Om)\ge (\xi_1+f_{0})\circ h_{-1}^{-1}(W_{-1}^{\phi'}(\Om)).
\end{equation}
 Brendle, Hung and Wang in \cite{BHW} proved $W_1^{\phi'}-W_{-1}^{\phi'}\ge(h_1-h_{-1})\circ f_1^{-1}(W_1(\Om))$ in hyperbolic space, i. e.
\begin{equation}\label{5.11}
	\int_{\p\Om}\Phi p_1d\mu+W_0+(n-1)W_2\ge W_{-1}^{\phi'}+(h_1-h_{-1})\circ f_1^{-1}(W_1(\Om)).
\end{equation}
If we have the inequality $W_2(\Om)\ge f_2\circ f_1^{-1}(W_1(\Om))$, by (\ref{5.10}), the similar inequality of (\ref{5.11}) will hold in sphere.

In \cite{HL2}, Hu and Li derive the following inequality in hyperbolic space:
$$W_{-1}^{\phi'}(\Om)\ge h_{-1}\circ f_0^{-1}(W_0(\Om)).$$
We derive a family of inverse inequalities in sphere, which is a surprising result.
We can also get the Alexandrov-Fenchel inequalities (\ref{5.15}), whose have been proved in \cite{CGL,CS}.
\begin{theorem}\label{t5.9}
Let $\Om$ be a star-shaped domain with smooth boundary $M$ in $\mS^{n+1}$. There holds
\begin{equation}
	W_{0}(\Om)\ge f_{0}\circ h_{-1}^{-1}(W_{-1}^{\phi'}(\Om)).
\end{equation}
Let $\Om$ be a strictly convex domain with smooth boundary $M$ in $\mS^{n+1}$. If $1\le m\le n$, there holds
	\begin{equation}\label{5.14}
		W_{m}(\Om)\ge f_{m}\circ h_{-1}^{-1}(W_{-1}^{\phi'}(\Om)),
	\end{equation}
and if $1\le m\le n-2$,
	\begin{equation}\label{5.15}
	W_{m}(\Om)\ge f_{m}\circ f_{0}^{-1}(W_{0}(\Om)).
\end{equation}
These equalities hold if and only if $\Om$ is a geodesic ball centered at the origin.
\end{theorem}
\textbf{Remark: } By using (\ref{5.15}), Chen and Sun derived $W_k\ge f_k\circ f_{k-2}^{-1}(W_{k-2})$ in \cite{CS}.

To prove these inequalities, we shall study (\ref{x1.7}) in sphere by technical reasons. (\ref{x1.7}) with $F=p_1$ has been investigated in \cite{GL} and \cite{CGL}. However, the $C^2$ estimates of (\ref{x1.7}) is still missing if $F\ne p_1$. We pick a suitable curvature function and discover a special algebraic property of this curvature function to derive the flow's result. Then these inequalities can be derived by the monotone quantities along this flow and the convergence of this flow.

The rest of this paper is organized as follows. We first recall some notations and known results in Section 2 for later use. In Section 3, we give the motivations for considering (\ref{x1.7}). We also derive an interesting algebra property, which we hope that might be useful to study other problems. Meanwhile, we derive the results of flow (\ref{x1.7}) in sphere by this interesting algebra property. In Section 4, we obtain some monotone quantities along the flows (\ref{x1.4}) and (\ref{x1.7}) respectively. As application, the inequalities in main theorems are proved. Finally, we raise some questions that could be considered.

\section{Preliminary}
\subsection{Intrinsic curvature}
We now state some general facts about hypersurfaces, especially those that can be written as graphs. The geometric quantities of ambient spaces will be denoted by $(\bar{g}_{\alpha\beta})$, $(\bar{R}_{\alpha\beta\gamma\delta})$ etc, where Greek indices range from $0$ to $n$. Quantities for $M$ will be denoted by $(g_{ij})$, $(R_{ijkl})$ etc., where Latin indices range from $1$ to $n$.

Let $\nabla$, $\bar\nabla$ and $D$ be the Levi-Civita connection of $g$, $\bar g$ and the Riemannian metric $e$ of $\mathbb S^n$  respectively. All indices appearing after the semicolon indicate covariant derivatives. The $(1,3)$-type Riemannian curvature tensor is defined by
\begin{equation}\label{2.1}
	R(U,Y)Z=\nabla_U\nabla_YZ-\nabla_Y\nabla_UZ-\nabla_{[U,Y]}Z,
\end{equation}
or with respect to a local frame $(e_i)$,
\begin{equation}\label{2.2}
	R(e_i,e_j)e_k={R_{ijk}}^{l}e_l,
\end{equation}
where we use the summation convention (and will henceforth do so). The coordinate expression of (\ref{2.1}), the so-called Ricci identities, read
\begin{equation}\label{2.3}
	Y_{;ij}^k-Y_{;ji}^k=-{R_{ijm}}^kY^m
\end{equation}
for all vector fields $Y=(Y^k)$. We also denote the $(0,4)$ version of the curvature tensor by $R$,
\begin{equation}\label{2.4}
	R(W,U,Y,Z)=g(R(W,U)Y,Z).
\end{equation}
\subsection{Extrinsic curvature}
For a hypersurface $X: M\longrightarrow\mathbb{K}^{n+1}$, the induced geometry of $M$ is governed by the following relations. The second fundamental form $h=(h_{ij})$ is given by the Gaussian formula
\begin{equation}\label{2.5}
	\bar\nabla_ZY=\nabla_ZY-h(Z,Y)\nu,
\end{equation}
where $\nu$ is a local outer unit normal field. Note that here (and in the rest of the paper) we will abuse notation by disregarding the necessity to distinguish between a vector $Y\in T_pM$ and its push-forward $X_*Y\in T_p\mathbb{K}^{n+1}$. The Weingarten endomorphism $A=(h_j^i)$ is given by $h_j^i=g^{ki}h_{kj}$, and the Weingarten formula
\begin{equation}\label{2.6}
	\bar\nabla_Y\nu=A(Y),
\end{equation}
holds there, or in coordinates
\begin{equation}\label{2.7}
	\nu_{;i}^\alpha=h_i^kX_{;k}^\alpha.
\end{equation}
We also have the Codazzi equation in $\mathbb{K}^{n+1}$
\begin{equation}\label{2.8}
	\nabla_Wh(Y,Z)-\nabla_Zh(Y,W)=-\bar{R}(\nu,Y,Z,W)=0
\end{equation}
or
\begin{equation}\label{2.9}
	h_{ij;k}-h_{ik;j}=-\bar R_{\alpha\beta\gamma\delta}\nu^\alpha X_{;i}^\beta X_{;j}^\gamma X_{;k}^\delta=0,
\end{equation}
and the Gauss equation
\begin{equation}\label{2.10}
	R(W,U,Y,Z)=\bar{R}(W,U,Y,Z)+h(W,Z)h(U,Y)-h(W,Y)h(U,Z)
\end{equation}
or
\begin{equation}\label{2.11}
	R_{ijkl}=\bar{R}_{\alpha\beta\gamma\delta}X_{;i}^\alpha X_{;j}^\beta X_{;k}^\gamma X_{;l}^\delta+h_{il}h_{jk}-h_{ik}h_{jl},
\end{equation}
where
\begin{equation}\label{2.12}
	\bar{R}_{\alpha\beta\gamma\delta}=-K(\bar{g}_{\alpha\gamma}\bar{g}_{\beta\delta}-\bar{g}_{\alpha\delta}\bar{g}_{\beta\gamma}).
\end{equation}
We give the gradient and hessian of the support function $u$ under the induced metric $g$ on $M$.
\begin{lemma}\label{l2.2}
	The support function $u$ satisfies
	\begin{equation}\label{2.13}
		\begin{split}
			\nabla_iu=&g^{kl}h_{ik}\nabla_l\Phi,  \\
			\nabla_i\nabla_ju=&g^{kl}\nabla_kh_{ij}\nabla_l\Phi+\phi'h_{ij}-(h^2)_{ij}u,
		\end{split}
	\end{equation}
	where $(h^2)_{ij}=g^{kl}h_{ik}h_{jl}.$
	$\Phi$ satisfies
	\begin{equation*}
		\nn_i\nn_j\Phi=\phi'g_{ij}-h_{ij}u.
	\end{equation*}
\end{lemma}
The proof of Lemma \ref{l2.2} can be seen in \cite{GL,BLO,JL}. Let $M(t)$ be a smooth family of closed hypersurfaces in $\mK^{n+1}$. Let $X(\cdot,t)$ denote a point on $M(t)$. In general, we have the following evolution property.
\begin{lemma}\label{l3.2}
	Let $M(t)$ be a smooth family of closed hypersurfaces in $\mK^{n+1}$ evolving along the flow
\begin{equation}\label{2.131}
\p_tX=\mathscr{F}\nu,
\end{equation}
	where $\nu$ is the unit outward normal vector field and $\mathscr{F}$ is a function defined on $M(t)$. Then we have the following evolution equations.
	\begin{equation}\label{3.4}
		\begin{split}
			\p_tg_{ij}&=2\mathscr{F}h_{ij},\\
			\p_t\nu&=-\nn\mathscr{F},\\
			\p_td\mu_g&=\mathscr{F}Hd\mu_g,\\
			\p_th_{ij}&=-\nn_i\nn_j\mathscr{F}+\mathscr{F}(h^2)_{ij}-K\mathscr{F}g_{ij},\\
			\p_tu&=\phi'\mathscr{F}-<\nn \Phi,\nn\mathscr{F}>,
		\end{split}
	\end{equation}
	where $d\mu_g$ is the volume element of the metric $g(t)$.
\end{lemma}
\begin{proof}
	Proof is standard, see for example, \cite{HG}.
\end{proof}
\subsection{Graphs in $\mathbb{K}^{n+1}$}
Let $(M,g)$ be a hypersurface in $\mathbb{K}^{n+1}$ with induced metric $g$. We now give the local expressions of the induced metric, second fundamental form, Weingarten curvatures etc when $M$ is a graph of a smooth and positive function $\rho(z)$ on $\mathbb{S}^n$. Let $\p_1,\cdots,\p_n$ be a local frame along $M$ and $\p_\rho$ be the vector field along radial direction. Then the support function, induced metric, inverse metric matrix, second fundamental form can be expressed as follows (\cite{GL}).
\begin{align*}
	u &= \frac{\phi^2}{\sqrt{\phi^2+|D\rho|^2}},\;\; \nu=\frac{1}{\sqrt{1+\phi^{-2}|D\rho|^2}}(\frac{\p}{\p\rho}-\phi^{-2}\rho_i\frac{\p}{\p x_i}),   \\
	g_{ij} &= \phi^2e_{ij}+\rho_i\rho_j,  \;\;   g^{ij}=\frac{1}{\phi^2}(e^{ij}-\frac{\rho^i\rho^j}{\phi^2+|D\rho|^2}),\\
	h_{ij} &=\(\sqrt{\phi^2+|D\rho|^2}\)^{-1}(-\phi D_iD_j\rho+2\phi'\rho_i\rho_j+\phi^2\phi'e_{ij}),\\
	h^i_j &=\frac{1}{\phi^2\sqrt{\phi^2+|D\rho|^2}}(e^{ik}-\frac{\rho^i\rho^k}{\phi^2+|D\rho|^2})(-\phi D_kD_j\rho+2\phi'\rho_k\rho_j+\phi^2\phi'e_{kj}),
\end{align*}
where $e_{ij}$ is the standard spherical metric. It will be convenient if we introduce a new variable $\gamma$ satisfying $$\frac{d\gamma}{d\rho}=\frac{1}{\phi(\rho)}.$$
Let $\omega:=\sqrt{1+|D\gamma|^2}$, one can compute the unit outward normal $$\nu=\frac{1}{\omega}(1,-\frac{\gamma_1}{\phi},\cdots,-\frac{\gamma_n}{\phi})$$ and the general support function $u=\la V,\nu\ra=\frac{\phi}{\omega}$. Moreover,
\begin{align}\label{2.14}
	g_{ij} &=\phi^2(e_{ij}+\gamma_i\gamma_j), \;\; g^{ij}=\frac{1}{\phi^2}(e^{ij}-\frac{\gamma^i\gamma^j}{\omega^2}),\notag\\
	h_{ij} &=\frac{\phi}{\omega}(-\gamma_{ij}+\phi'\gamma_i\gamma_j+\phi'e_{ij}),\notag\\
	h^i_j &=\frac{1}{\phi\omega}(e^{ik}-\frac{\gamma^i\gamma^k}{\omega^2})(-\gamma_{kj}+\phi'\gamma_k\gamma_j+\phi'e_{kj})\notag\\
	&=\frac{1}{\phi\omega}(\phi'\delta^i_j-(e^{ik}-\frac{\gamma^i\gamma^k}{\omega^2})\gamma_{kj}).
\end{align}
Covariant differentiation with respect to the spherical metric is denoted by indices.

There is also a relation between the second fundamental form and the radial function on the hypersurface. Let $\widetilde{h}=\phi'\phi e$. Then
\begin{equation}\label{2.15}
	\omega^{-1}h=-\nn^2\rho+\widetilde{h}
\end{equation}
holds; cf. \cite{GC2}. Since the induced metric is given by
$$g_{ij}=\phi^2 e_{ij}+\rho_i\rho_j, $$
we obtain
\begin{equation}\label{2.16}
	\omega^{-1}h_{ij}=-\rho_{;ij}+\frac{\phi'}{\phi}g_{ij}-\frac{\phi'}{\phi}\rho_{i}\rho_{j}.
\end{equation}

It is known (\cite{GC2}) if a closed hypersurface in $\mathbb{K}^{n+1}$ which is a radial graph over $\mathbb{S}^n$ and satisfies the flow equation (\ref{2.131}),
then the evolution of the scalar function $\rho=\rho(X(z,t),t)$ satisfies
$$\p_t\rho=\mathscr{F}\omega.$$
Lastly, we have the connection between $\vert\nn\rho\vert$ and $\vert D\gamma\vert$.
\begin{lemma}\label{l2.3} \cite{DL}
	If $M$ is a star-shaped hypersurface, there holds $\vert\nn\rho\vert^2=1-\frac{1}{\omega^2}$.
\end{lemma}

\subsection{Elementary symmetric functions}
We review some properties of elementary symmetric functions. See \cite{HGC} for more details.

In Section 1 we give the definition of elementary symmetric functions. The definition can be extended to symmetric matrices. Let $A\in Sym(n)$ be an $n\times n$ symmetric matrix. Denote by $\k=\k(A)$ the eigenvalues of $A$. Set $p_m(A)=p_m(\k(A))$. We have
\begin{equation*}
	p_m(A)=\frac{(n-m)!}{n!}\delta^{j_1\cdots j_m}_{i_1\cdots i_m}A_{j_1}^{i_1}\cdots A_{j_m}^{i_m}, \qquad m=1,\cdots,n,
\end{equation*}
where $\delta_{i_1\cdots i_m}^{j_1\cdots j_m}$ is generalized Kronecker delta defined by
$$\delta_{i_1\cdots i_m}^{j_1\cdots j_m}=\det\(\(\delta_{i_k}^{j_l}\)_{m\times m}\), \qquad 1\le k,l,i_k,j_l\le m.$$
\begin{lemma}\label{l2.4}
	Denote $(p_m)^j_i=\frac{\p p_m}{\p A_{j}^{i}}$. Then we have
	\begin{align}
		(p_m)^j_iA_{j}^{i}&=mp_m,\\
		(p_m)^j_i\delta_j^i&=mp_{m-1},\\
		(p_m)^j_i(A^2)_{j}^{i}&=np_1p_m-(n-m)p_{m+1},
	\end{align}
	where $(A^2)_{j}^{i}=\sum_{k=1}^nA_{j}^{k}A_{k}^i$.
\end{lemma}
Using Lemma \ref{l2.2}, we have an important formula,
$$p_m^{ij}\nn_i\nn_j\Phi=p_m^{ij}(\phi'g_{ij}-h_{ij}u)=m(\phi'p_{m-1}-up_m).$$
This formula plays an important role in our subsequent proof. We denote $\sigma_k=C_n^kp_k$, i. e. the $k$th elementary symmetric function $\sigma_k$ is defined by
$$\sigma_k(\k_1,...,\k_n)=\sum_{1\leq i_1<\cdot\cdot\cdot<i_k\leq n}\k_{i_1}\cdot\cdot\cdot\k_{i_k}.$$
\begin{lemma}\label{l2.5}\cite{CGL}
	If $\k\in\Gamma_m^+=\{\k\in\mR^n:p_i(\k)>0, i=1,\cdots,m\}$, we have the following Newton-MacLaurin inequalities.
	\begin{align}
		p_{m}(\k)p_{k-1}(\k)&\le p_k(\k)p_{m-1}(\k), \quad 1\le k<m\le n,\\
		p_1\ge p_2^{\frac{1}{2}}\ge\cdots&\ge p_m^{\frac{1}{m}},\qquad 1<m\le n.
	\end{align}
	Equality holds if and only if $\k_1=\cdots=\k_n$.
	
Suppose that $\k_1\ge\cdots\ge\k_m\ge\cdots\ge\k_n$, then we have
\begin{align}
\k_1\ge\cdots\ge\k_m>0,\quad &\sigma_m(\k)\le C_n^m\k_1\cdots\k_m;\label{x2.26}\\
\sigma_m(\k)\ge \k_1\cdots\k_m, \quad &\text{ if }\k\in\Gamma_{m+1}^+.\label{x2.27}
\end{align}
\end{lemma}
Let us denote by $\sigma_{k}(\k|i)$ the sum of the terms of $\sigma_k(\k)$ not containing the factor $\k_i$. Then the following identities hold.
\begin{proposition}\label{p2.6}
	\cite{HGC} We have, for any $k=0,\cdots,n$, $i=1,\cdots,n$ and $\k\in\mR^n$,
	\begin{align}
		\frac{\p\sigma_{k+1}}{\p\k_i}(\k)&=\sigma_{k}(\k|i),\\
		\sigma_{k+1}(\k)&=\sigma_{k+1}(\k|i)+\k_i\sigma_{k}(\k|i),\\
		\sum_{i=1}^n\sigma_{k}(\k|i)&=(n-k)\sigma_k(\k),\\
		\sum_{i=1}^n\k_i\sigma_{k}(\k|i)&=(k+1)\sigma_{k+1}(\k),\\
		\sum_{i=1}^n\k_i^2\sigma_{k}(\k|i)&=\sigma_1(\k)\sigma_{k+1}(\k)-(k+2)\sigma_{k+2}(\k).
	\end{align}
\end{proposition}

Recall in \cite{R1,R2}, Reilly proved a formula for the $r$-th Newton operator $T_r(h)$ of the second fundamental form in general Riemannian manifolds, see Proposition 1 in \cite{R2}. In space form, since second fundamental form is Codazzi, Reilly's formula is equivalent to say the Newton operator $T_r$ is divergent free. More specifically, in local coordinates, the Newton operator $T_r$ of $h^i_j$ is a symmetric matrix $(T_r)_i^j=\dfrac{\p\sigma_{r+1}}{\p h^i_j}=:(\sigma_{r+1})_i^j$, and $$\nn_j(T_r)_i^j=0.$$
Using this property, we have the Minkowski formula.
\begin{proposition}\label{xp}
	Let $M$ be a closed hypersurface in $\mK^{n+1}$. Then, for $k=0,1,\cdots,n-1,$
	\begin{equation}\label{x2}
		\int_{M}p_{k+1}(\k)u=\int_{M}\phi'(\rho)p_k(\k),
	\end{equation}
	where we use the convention that $p_0\equiv1.$
\end{proposition}
\begin{proof}
	Applying divergence theorem, we obtain
	\begin{align*}
		0=\int_{M}\nn_i\((T_k)^i_j\nn^j\Phi\)=\int_M(n-k)\phi'\sigma_k-(k+1)\sigma_{k+1}u.
	\end{align*}
\end{proof}
Finally, we give some properties of curvature function $F$ in flows (\ref{x1.4}) and (\ref{x1.7}). It is useful to consider $F$ as a function of the symmetric matrix $[a_{j}^i]$ as well as $(\mu_1,\cdots,\mu_n)$, where $\mu_1,\cdots,\mu_n$ are the eigenvalues of the matrix $[a_{j}^i]$, i. e.
\begin{equation}\label{3.1}
	F([a_{j}^i])=F(\mu_1,\cdots,\mu_n),
\end{equation}
It is not difficult to see that the eigenvalues of $[F^{j}_i]=[\frac{\p F}{\p a_{j}^i}]$ are $\frac{\p F}{\p\mu_1},\cdots,\frac{\p F}{\p\mu_n}$.

\begin{proposition}\label{p2.8}
(1) Suppose $F$ is strictly increasing, homogeneous of degree $1$, and strictly positive on $\Gamma$ with $F(1,\cdots,1)=1$.
 If $F$ is concave, then $\sum_{i=1}^nF^i_i\ge1$;
If $F$ is inverse concave, then $F^j_i(h^2)_j^i\ge F^2.$

(2) If $F=\frac{p_k}{p_{k-1}}$, $1\le k\le n$, then on ${\Gamma}_k$,
$$1\le\sum_{i=1}^nF^i_i\le k,\qquad F^2\le F^j_i(h^2)_j^i\le(n-k+1)F^2.$$
\end{proposition}
\begin{proof}
(1) has been proved in \cite{AMZ} Lemma 4 and Lemma 5. ($2$) was first observed by Brendle, Guan and Li. We give a simple proof here. At one point, we can choose normal coordinates to let $(F^j_i)$ and $(h_j^i)$ be diagonal. Then by Lemma \ref{l2.4}, we have
\begin{equation*}
\sum_{i}F^i_i=\dfrac{\sum_{i}p_k^ip_{k-1}-\sum_{i}p_{k-1}^ip_{k}}{p_{k-1}^2}=k-(k-1)\dfrac{p_{k-2}p_{k}}{(p_{k-1})^2}.
\end{equation*}
By Lemma \ref{l2.5}, we get $\sum_{i}F^i_i\ge 1.$ Because $\dfrac{p_{k-2}p_{k}}{(p_{k-1})^2}$ is nonnegative on ${\Gamma}_k$, we have $\sum_{i=1}^nF^i_i\le k$. Another inequality can be proved in the same way.
\end{proof}

\section{Flow (\ref{x1.7}) in sphere with $F=p_k^\frac{1}{k}$}
In this section, we introduce the locally constrained flow (\ref{x1.7}), which is an extension of the mean curvature type flow studied in \cite{GL}. Let $X_0:M^n\rightarrow\mS^{n+1}$ be a smooth embedding such that $M_0$ is a closed, star-shaped hypersurface in $\mS^{n+1}$. We consider the smooth family of immersions $X:M^n\times[0,T)\rightarrow\mS^{n+1}$ satisfying the following evolution equations:
\begin{equation}\label{4.1}
	\frac{\p}{\p t}X(x,t)=\(\phi'-uF\)\nu(x,t).
\end{equation}

\subsection{The motivation of studying flow (\ref{4.1}) in sphere with $F=p_k^\frac{1}{k}$}

Since the $C^2$ estimates of flow (\ref{4.1}) in sphere with $F=\frac{p_k}{p_{k-1}}$ is still open, we shall study flow (\ref{4.1}) in sphere with $F=p_k^\frac{1}{k}$. The Alexandrov-Fenchel inequalities in sphere can also be derived by this flow. Specially,
\begin{lemma}
	Suppose $M_t$ is a smooth, closed, strictly $k$-convex solution to the flow (\ref{4.1}) with $F=p_k^\frac{1}{k}$ in sphere, $2\le k\le n$. Then the following hold:
	
	$(\rmnum1)$ $W_0$ is monotone increasing and $W_0$ is a constant function if and only if $M_t$ is a slice for each $t$;
	
	$(\rmnum2)$ $W_m$ is monotone decreasing with $k-1\le m\le k$ and $W_m$ is a constant function if and only if $M_t$ is a slice for each $t$;
	
	$(\rmnum3)$ $W_{-1}^{\phi'}$ is monotone increasing and $W_{-1}^{\phi'}$ is a constant function if and only if $M_t$ is a slice for each $t$.
\end{lemma}
However, the longtime existence of flow (\ref{4.1}) with $F=p_k^\frac{1}{k}$ is also hard to be derived if $k\ge3$. In this section, we derive this result if $k=2$ and the initial hypersurface is strictly convex, which can also be used to proved Theorems \ref{t5.7} and \ref{t5.9} by the following lemma.
\begin{lemma}	
	Suppose $M_t$ is a smooth, closed, strictly convex solution to the flow (\ref{4.1}) with $F=p_2^\frac{1}{2}$ in sphere. Then the following hold:
	
	$(\rmnum1)$ $W_0$ is monotone increasing and $W_0$ is a constant function if and only if $M_t$ is a slice for each $t$;
	
	$(\rmnum2)$ $W_m$ is monotone decreasing with $1\le m\le n-1$ and $W_m$ is a constant function if and only if $M_t$ is a slice for each $t$;
	
	$(\rmnum3)$ $W_{-1}^{\phi'}$ is monotone increasing and $W_{-1}^{\phi'}$ is a constant function if and only if $M_t$ is a slice for each $t$;
	
	$(\rmnum4)$ $\int_{M_t}\Phi p_kd\mu_t+kW_{k-1}(\Om_t)$ is monotone decreasing with $1\le k\le n-2$ and $\int_{M_t}\Phi p_kd\mu+kW_{k-1}(\Om_t)$ is a constant function if and only if $M_t$ is a slice for each $t$.
\end{lemma}
The proof of above two lemmas is similar to Lemma \ref{l5.6}, and we no longer give a detailed proof.
\subsection{The result of flow (\ref{x1.7}) in sphere}

By Section 2.3, we only need to consider the following parabolic initial value problem on $\mathbb{S}^n$,
\begin{equation}\label{x2.17}
	\begin{cases}
		\p_t\rho&=\(\phi'-uF\)\omega, \;\;(z,t)\in\mathbb{S}^n\times [0,\infty),\\
		\rho(\cdot,0)&=\rho_0.
	\end{cases}
\end{equation}
Equivalently, the equation for $\gamma$ satisfies
\begin{equation}\label{x2.18}
	\p_t\gamma=\(\phi'-uF\)\frac{\omega}{\phi}.
\end{equation}

To prove Theorem \ref{t4.1}, we need the following lemmas.
\begin{lemma}\label{xl4.2}
	Let $\rho(x,t)$, $t\in[0,T)$, be a smooth, star-shaped solution to (\ref{x2.17}) with initial hypersurface $M_0$ in $\mathbb S^{n+1}$. Then
	\begin{equation}\label{4.2}
		\min_{z\in \mS^n}\rho(z,0)\leq \rho(z,t)\leq \max_{z\in \mS^n}\rho(z,0),\quad \forall(z,t)\in\mS^n\times[0,T).
	\end{equation}
\end{lemma}
\begin{proof}
	This proof is derived directly by the maximal principle. We omit this proof here.
\end{proof}
We give some evolution equations of flow (\ref{4.1}).
\begin{lemma}\label{xl4.3}
	Let $M(t)$ be a smooth family of closed hypersurfaces in $\mS^{n+1}$ evolving along flow (\ref{4.1}). Then we have the following evolution equations.
	\begin{align}
		\cL u=&-\nn\Phi\nn\phi'+F\nn\Phi\nn u+(\phi'-2uF)\phi'+u^2F^{ij}(h^2)_{ij},\label{4.3}\\
		\cL F=&2F^{ij}\nn_iu\nn_jF+F\nn\Phi\nn F-(F^{ij}(h^2)_{ij}-F^2)\phi'+uF(\sum F^{ii}-1),\label{4.4}
	\end{align}
	where $\cL=\p_t-uF^{ij}\nn_i\nn_j$.
\end{lemma}
\begin{proof}
	The calculation can be seen in \cite{CGL}. Note that we should replace $c_{n,k}$ by $1$ since the form of the flow in \cite{CGL} is different from (\ref{4.1}).
\end{proof}
Then we have the $C^1$ estimates by the above lemma.
\begin{lemma}\label{xl4.4}
	Let $M_0$ be a star-shaped, strictly 2-convex hypersurface in $\mS^{n+1}$. Along the flow (\ref{4.1}) with $F=p_2^\frac{1}{2}$, the spatial minimum of $u$ is increasing. Therefore, we have the lower bound of $u$ and $|D\gamma|\le C_{1}$, where $C_{1}$ only depends on the initial hypersurface.
\end{lemma}
\begin{proof}
	At the spatial minimum point $P$ of $u$, we have
	$$\nn u=0.$$
	By $F^j_i(h^2)^i_j\ge F^2$ and Lemma \ref{xl4.3}, we have
	$$\cL u\ge |\nn\Phi|^2+(\phi'-uF)^2\ge0.$$
	This implies that the spatial minimum of $u$ is increasing. Meanwhile, we have $|D\gamma|\le C_{1}$ by $u=\frac{\phi}{\om}$.
\end{proof}
To derive the bounds of $F$ and the $C^2$ estimates, we find an interesting algebra property of $p_2^\frac{1}{2}$. We denote $f\sim g$ by $\frac{1}{C}g\le f\le Cg$ for some constant $C>0$.
If $F=\frac{p_{k+1}}{p_{k}}$ with $1\le k\le n-2$ and the hypersurface is strictly convex, the algebra property $(F^{ii}-1)\sim (\frac{F^{ij}(h^2)_{ij}}{F^2}-1)$ has been proved in \cite{CGL}. 
If $F=p_2^\frac{1}{2}$, both $F^{ii}$ and $\frac{F^{ij}(h^2)_{ij}}{F^2}$ are unbounded in some cases. We also find an interesting algebra property that $(F^{ii}-1)\sim (\frac{F^{ij}(h^2)_{ij}}{F^2}-1)$ for $F=p_2^\frac{1}{2}$. We hope that might be useful to study other problems.
\begin{lemma}\label{xl4.5}
	Let $\k=(\k_1,\cdots,\k_n)\in\Gamma_2^+$ and $F=p_2^\frac{1}{2}$. Then we have
	$$\frac{F^{ij}(h^2)_{ij}}{F^2}-1\ge\frac{n}{2}(F^{ii}-1).$$
\end{lemma}
\begin{proof}
	By Lemma \ref{l2.4} and direct computation, we derive
	\begin{align*}
		\frac{F^{ij}(h^2)_{ij}}{F^2}-1&=\dfrac{np_1p_2-(n-2)p_3-2p_2^\frac{3}{2}}{2p_2^\frac{3}{2}}=\frac{n-2}{2}\frac{p_1p_2-p_3}{p_2^\frac{3}{2}}+\frac{p_1-p_2^\frac{1}{2}}{p_2^\frac{1}{2}}\\
		&\ge \frac{n}{2}\frac{p_1-p_2^\frac{1}{2}}{p_2^\frac{1}{2}}=\frac{n}{2}(F^{ii}-1),
	\end{align*}
	where we have used $p_2^\frac{3}{2}\ge p_3$ if $\kappa=(\kappa_1,\cdots,\kappa_n)\in\Gamma_2^+$.
\end{proof}
\begin{lemma}\label{xl4.6}
	Let $\k=(\k_1,\cdots,\k_n)\in\Gamma_n^+=\Gamma_+$ and $F=p_2^\frac{1}{2}$. Then we have
	$$\(\frac{F^{ij}(h^2)_{ij}}{F^2}-1\)\sim (F^{ii}-1).$$
\end{lemma}
\begin{proof}
	Note that
	$$\frac{F^{ij}(h^2)_{ij}}{F^2}-1=\frac{n-2}{2}\frac{p_1p_2-p_3}{p_2^\frac{3}{2}}+\frac{p_1-p_2^\frac{1}{2}}{p_2^\frac{1}{2}},\quad F^{ii}-1=\frac{p_1-p_2^\frac{1}{2}}{p_2^\frac{1}{2}}.$$
	It suffices to prove $\frac{(n-2)\sigma_1\sigma_2-3n\sigma_3}{\sigma_2}\le C(p_1-p_2^\frac{1}{2})$ by $p_k=(C_n^k)^{-1}\sigma_k$. By Proposition \ref{p2.6}, we get
	\begin{equation}\label{4.5}
		\begin{split}
			(n-2)\sigma_1\sigma_2-3n\sigma_3=&(\k_1+\cdots+\k_n)\sum_{i}\sigma_2(\k|i)-n\sum_{i}\k_i\sigma_2(\k|i)\\
			=&\sum_j\k_j\sum_{i}\sigma_2(\k|i)-n\sum_{i}\k_i\sigma_2(\k|i)\\
			=&\sum_{j}\sum_{i}(\k_j-\k_i)\sigma_2(\k|i)\\
			=&\sum_{i<j}(\k_j-\k_i)(\sigma_2(\k|i)-\sigma_2(\k|j))\\
			=&\sum_{i<j}(\k_j-\k_i)(\k_j\sigma_1(\k|ij)-\k_i\sigma_1(\k|ij))\\
			=&\sum_{i<j}(\k_j-\k_i)^2\sigma_1(\k|ij).
		\end{split}
	\end{equation}
	Let $\k_1\ge\cdots\ge\k_n$. By (\ref{x2.26}), (\ref{x2.27}) and (\ref{4.5}), we can obtain
	\begin{equation}\label{4.6}
		\begin{split}
			&\frac{(n-2)\sigma_1\sigma_2-3n\sigma_3}{\sigma_2}\\
			=&\frac{\sum_{i<j}(\k_j-\k_i)^2\sigma_1(\k|ij)}{\sigma_2}\\
			\le &(n-2)\(\frac{\sum_{2\le i<j}(\k_j-\k_i)^2\k_1}{\k_1\k_2}+\frac{\sum_{ i<2<j}(\k_j-\k_i)^2\k_2}{\k_1\k_2}+\frac{\sum_{ i<j\le2}(\k_j-\k_i)^2\k_3}{\k_1\k_2}\)\\
			\le&C_{2}\(\frac{(\k_2-\k_n)^2}{\k_2}+\frac{2(\k_1-\k_n)^2}{\k_1}\)\le3C_{2}\frac{(\k_1-\k_n)^2}{\k_1},
		\end{split}
	\end{equation}
	where we have used $\frac{(\k_2-\k_n)^2}{\k_2}\le\frac{(\k_1-\k_n)^2}{\k_1}.$ In fact,
	\begin{equation*}
		\frac{(\k_2-\k_n)^2}{\k_2}-\frac{(\k_1-\k_n)^2}{\k_1}=(\k_2+\frac{\k_n^2}{\k_2})-(\k_1+\frac{\k_n^2}{\k_1})\le0
	\end{equation*}
	by the property of the function $y=x+\frac{\k_n^2}{x}$ and $0<\k_n\le\k_2\le\k_1$. On the other hand, by $\sigma_1\le n\k_1$ and $(n-1)\sigma_1^2-2n\sigma_2=\sum_{ i<j}(\k_i-\k_j)^2$,
	\begin{equation}\label{4.7}
		\begin{split}
			p_1-p_2^\frac{1}{2}&=\frac{p_1^2-p_2}{p_1+p_2^\frac{1}{2}}\ge C(n)\frac{(n-1)\sigma_1^2-2n\sigma_2}{\sigma_1}\\
			&\ge C(n)\frac{\sum_{ i<j}(\k_i-\k_j)^2}{\k_1}\ge C(n)\frac{(\k_1-\k_n)^2}{\k_1}.
		\end{split}
	\end{equation}
	Combining (\ref{4.6}) and (\ref{4.7}), we complete this proof.
\end{proof}
Due to Lemma \ref{xl4.5} and \ref{xl4.6}, we obtain the uniform bound of $F$ as follows.
\begin{lemma}\label{xl4.7}
	Let $M_0$ be a star-shaped, strictly 2-convex hypersurface in $\mS^{n+1}$. Along the flow (\ref{4.1}) with $F=p_2^\frac{1}{2}$, we have
	$$F\le C_{3},$$ where $C_{3}$ only depends on the initial hypersurface.
\end{lemma}
\begin{proof}
	At the spatial maximum point of $F$, we have $\nn F=0$. Thus from Lemma \ref{xl4.3}, we can get
	$$\cL F=-\phi'F^2(\frac{F^{ij}(h^2)_{ij}}{F^2}-1)+uF(\sum F^{ii}-1).$$
	By Lemma \ref{xl4.5} and the maximum principle, the proof is completed.
\end{proof}
\begin{lemma}\label{xl4.8}
	Let $M_0$ be a strictly convex hypersurface in $\mS^{n+1}$. Along the flow (\ref{4.1}) with $F=p_2^\frac{1}{2}$, we have
	$$F\ge C_{4},$$ where $C_{4}$ only depends on the initial hypersurface.
\end{lemma}
\begin{proof}
	Using Lemma \ref{xl4.6}, this proof is similar to Lemma \ref{xl4.7}.
\end{proof}
Then we can start getting the $C^2$ estimates.
\begin{lemma}\label{xl4.9}
	Let $X(\cdot,t)$ be the smooth, strictly convex solution to the flow (\ref{4.1}) with $F=p_2^\frac{1}{2}$ which encloses the origin for $t\in[0,T)$. Then there is a positive constant $C_{5}$ depending on the initial hypersurface,  such that the principal curvatures of $X(\cdot,t)$ are uniformly bounded from below $$\k_i(\cdot,t)\geq C_{5} \text{ \qquad } \forall1\le i\le n,$$
	and hence, are compactly contained in $\Gamma_n^+$, in view of Lemma \ref{xl4.8}.
\end{lemma}
\begin{proof}
	We shall prove that the principal curvature radii $\l_i=\frac{1}{\k_i}$  is bounded from upper by a positive constant. The principal curvature radii of $M_t$ are the eigenvalues of $\{h^{il}g_{lj}\}$. Note that $p_2^\frac{1}{2}$ is uniformly continuous on the convex cone $\bar\Gamma_n^+$, the dual function of $F$, $(\frac{p_n}{p_{n-2}}(\l))^\frac{1}{2}|_{\p\Gamma_n^+}=0$, and $F$ is bounded from upper by a positive constant. If we have the upper bound of $\l_i$, Lemma \ref{xl4.8} imply that $\l_i$ remains in a fixed compact subset of $\Gamma_n^+$, which is independent of $t$. In other words, $\k_i$ remains in a fixed compact subset of $\Gamma_n^+$, which is independent of $t$. Therefore, next we only prove that $\l_i$ have an upper bound.
	
	We consider the following quantity
	$$\tilde\theta(x,t)=\log W(x,t)-\log u(x,t),$$
	where
	$$W(x,t)=\max\{h^{ij}(x,t)\xi_i\xi_j:g^{ij}\xi_i\xi_j=1\}.$$
	Assume that $\theta$ attains its maximum at $(x_0,t_0)$. At this point, we choose Riemannian normal coordinates, such that at this point we have
	\begin{equation*}
		g_{ij}=\delta_{ij},\quad h^i_j=h_{ij}=\k_i\delta_{ij}.
	\end{equation*}
	By a rotation, we may suppose that $W(x_0,t_0)=h^{ij}(x,t)\xi_i\xi_j$ with $\xi=(1,0,\cdots,0)$. After a standard discussion (refer to \cite{CS,DL} etc.), we only need to derive the upper bound of
	$$\theta(x,t)=\log(\l_1)-\log u,$$
	where  $\l_1=h^{11}g_{11}$.
	The evolution equation of $\theta$ has been calculated in \cite{CS}. We directly give the equation here without calculating. Note that we should replace $c_{n,k}$ in \cite{CS} by $1$ here.
	\begin{align}
		0\le& \p_t(\log\l_1-\log u)\notag\\
		\le& \frac{1}{\l_1u}\(-\l_1((2F^{ij}(h^2)_{ij}-F^{ii})u^2-u^2-\phi'uF+(\phi')^2+|\nn\phi'|)^2)\notag\\
		&+\frac{1}{2}F(\nn_1\Phi)^2-2u^2\l_1^2+u(\phi'+uF)\).\label{4.8}
	\end{align}
	Recall the proof in \cite{CS}.
	If $F=\frac{p_k}{p_{k-1}}$, by Proposition \ref{p2.8} and the upper bound of $F$, the upper bound of $\l_1$ can be derived by the maximum principle. However, if $F=p_2^\frac{1}{2}$, we don't have the upper bound of $F^{ii}$.  Thus we need to do more fine-grained processing.
	Thanks for Lemma \ref{xl4.5}, we have $F^{ij}(h^2)_{ij}\ge\frac{nF^2}{2}F^{ii}-\frac{(n-2)F^2}{2}.$ Meanwhile,
	$$F=F^{kk}h_{kk}\ge F^{kk}h_{11}$$
	implies $\l_1=\frac{1}{h_{11}}\ge \frac{F^{kk}}{F}$. By weighted Young's inequality $ab\le\eps a^p+\frac{(p\eps)^{-\frac{q}{p}}}{q}b^q$, we have
	$$F^{kk}\le\frac{F^{kk}}{F}+\frac{(\frac{3}{2})^{-{2}}}{3}F^2F^{kk}<\frac{F^{kk}}{F}+F^2F^{kk}\le \l_1+F^2F^{kk}.$$
	Then (\ref{4.8}) implies
	$$0\le -u^2\l_1^2+C_{6}\l_1+C_{7}$$
	by the upper bound of $F$ and the uniform bound of $u$, $\rho$. By the maximum principle we complete this proof.
\end{proof}
The estimates obtained in Lemmas \ref{xl4.2}, \ref{xl4.4}, \ref{xl4.9} only depend on the geometry of the initial data $M_0$. By Lemmas \ref{xl4.2}, \ref{xl4.4}, \ref{xl4.9}, we conclude that the equation (\ref{x2.17}) is uniformly parabolic. By the $C^0$ estimate (Lemma \ref{xl4.2}), the gradient estimate (Lemma \ref{xl4.4}), the $C^2$ estimate (Lemma \ref{xl4.9}) and the Krylov's and Nirenberg's theory \cite{KNV,LN}, we get the H$\ddot{o}$lder continuity of $D^2\rho$ and $\rho_t$. Then we can get higher order derivative estimates by the regularity theory of the uniformly parabolic equations. Hence we obtain the long time existence and $C^\infty$-smoothness of solutions for the flow (\ref{4.1}). The uniqueness of smooth solutions also follows from the parabolic theory.

Next we shall prove the solution of (\ref{4.1}) converges exponentially to a geodesic slice in the $C^\infty$-topology.
Firstly, we will show the solution of (\ref{4.1})  converges to a geodesic slice in the $C^\infty$-topology. In other words, when $t$ is large enough, different principal curvatures of the radial graph at the same point are comparable uniformly for arbitrary small $\eps$.
\begin{lemma}\label{l4.10}
	Let $\rho(x,t)$ be a smooth solution to (\ref{x2.17}) with
	$$||\rho||_{C^2(\mS^n)}\le C,\qquad \forall t.$$
	Then for any $\eps>0$, there exists a $T_1>0$, such that for any $t>T_1$,
	\begin{equation*}
		\max_{i<j,z\in\mS^n}|\k_i(z,t)-\k_j(z,t)|<\eps,
	\end{equation*}
	where $\k(z,t)$ are the principal curvatures of the radial graph $\rho(z,t)$ at $(z,t)\in\mS^n\times[0,\infty).$
\end{lemma}
\begin{proof}
	In order to complete the proof of Lemma \ref{l4.10}, all that we have to show is that each subsequential limit is a slice independent of the subsequence as $t\rightarrow\infty$.
	
	The evolution of $W_0(\Om)$ is
	\begin{align*}
		\frac{d}{dt}W_0(\Om_t)=&\int_{\p\Omega_t}(\phi'-up_2^\frac{1}{2})d\mu_t=\int_{\p\Omega_t}(\phi'-up_1)d\mu_t+\int_{\p\Omega_t}u(p_1-p_2^\frac{1}{2})d\mu_t\ge0.
	\end{align*}
	Since $W_0$ is obviously bounded, we have
	$$\int_{0}^\infty\int_{M_t}u(p_1-p_2^\frac{1}{2})d\mu dt<\infty,$$
	and hence
	$$\int_{M_t}u(p_1-p_2^\frac{1}{2})d\mu\rightarrow0.$$
	We have established the uniform a priori estimates for all the derivatives of $\rho$ both with respect to $z$ and $t$ for any order and also $0<\frac{1}{C_{8}}\le u\le C_{8}$ and $0<\frac{1}{C_{9}}\le p_1+p_2^\frac{1}{2}\le C_{9}$ uniformly, $W_0$ is uniformly continuous with respect to $t$ and $d\mu_t$ is also uniformly bounded from both up and below. This implies that
	$$0\le p_1-p_2^\frac{1}{2}\le \eps'.$$
	Thus
	$$\sum_{ i<j}(\k_i-\k_j)^2=C(n)(p_1^2-p_2)=C(n)(p_1+p_2^\frac{1}{2})(p_1-p_2^\frac{1}{2})\le C_{10}\eps'.$$
	Then for any $\eps>0$, there exists a $T_1>0$, such that for any $t>T_1$,
	\begin{equation*}
		\max_{i<j,z\in\mS^n}|\k_i(z,t)-\k_j(z,t)|\le\(\sum_{ i<j}(\k_i-\k_j)^2\)^\frac{1}{2}<\eps.
	\end{equation*}
\end{proof}
\begin{lemma}\label{l4.11}
	Let $\g(x,t)$ be a smooth solution to (\ref{x2.18}). Then there are positive constants $C_{11}$ and $C_{12}$ depending on the initial hypersurface,  such that
	$$	\max_{\mS^n}\vert D\g\vert^2\leq C_{11}e^{-C_{12}t},\;\;\forall t>T_1,$$
	where $T_1>0$ is defined as in Lemma \ref{l4.10}.
\end{lemma}
\begin{proof}
	Let $\eps,\delta>0$ be small enough constants to be determined and $T_1>0$ be defined as in Lemma \ref{l4.10}. For any $T>T_1,$ Consider the auxiliary function $O=\frac{1}{2}\vert D\g\vert^2$. Assume $O$ achieves maximum on the interval $[T,T+\delta]$ at point $P$. At $P$,
	\begin{gather*}
		D\omega=0,\\
		0=D_iO=\sum_{l}\g_{li}\g^{l},\\
		0\geq D_{ij}^2O=\sum_{l}\g_{i}^l\g_{lj}+\sum_{l}\g^{l}\g_{lij},
	\end{gather*}
	where $\g^l=\g_ke^{kl},\g_{i}^l=\g_{ik}e^{kl}$. We mention that $\g_{ij}$ is symmetric.
	
	At the critical point $P$, we can simplify the Weingarten tensor $h_j^i=\frac{1}{\om\phi}(-\g^i_{j}+\phi'\delta_j^i)$, where $\g^i_{j}=e^{ik}\g_{kj}$ and $[\g_{ij}]$ is symmetric matrix. By rotating the orthonormal coordinates, we can assume $\g_1=|D\g|$ and $\g_j=0$ for $2\le j\le n$. Then by critical point condition, $\g_{1j}=0$ for all $j=1,\cdots,n$. Now fixing the $i=1$ direction, we can diagonalize the rest $(n-1)\times(n-1)$ matrix. Finally, $\g_{ij}$ becomes a diagonal matrix with $\g_{11}=0$. Since the Weingarten tensor $h_j^i=\frac{1}{\om\phi}(-\g^i_{j}+\phi'\delta_j^i)$ at $P$, it is also diagonalized so the eigenvalues are $h_i^i$ which are also the principal curvatures. Thus at $P$, the principal curvature
	$$\k_i=h_i^i=\frac{1}{\om\phi}(-\g^i_{i}+\phi').$$
	By Lemma \ref{l4.10}, we have
	$$|\g_{ii}-\g_{jj}|\le\eps C_{13},$$
	where $C_{13}$ is a uniform constant and we have used the uniformly boundedness of $\rho$ and $\om$. Since $\g_{11}=0$, we have $|\g_{jj}|=|\g_{11}-\g_{jj}|\le\eps C_{13}$ for $2\le j\le n$.
	
	We rewrite the flow equation (\ref{x2.18}),
	\begin{equation}\label{4.11}
		\begin{split}
			\p_t\gamma=&\(\phi'-uF(\frac{1}{\phi\omega}(\phi'\delta^i_j-\widetilde{g}^{ik}\gamma_{kj}))\)\frac{\omega}{\phi}\\
			=&\dfrac{\om\phi'}{\phi}-\frac{F(\phi'\delta^i_j-\widetilde{g}^{ik}\gamma_{kj})}{\om\phi}=:G(\g,D\g,D^2\g),
		\end{split}
	\end{equation}
	where $\widetilde{g}^{ik}=e^{ik}-\frac{\gamma^i\gamma^k}{\omega^2}$. For convenience, we let $b^i_j=\phi'\delta^i_j-\widetilde{g}^{ik}\gamma_{kj}$ and $[b^i_j]$ is diagonal. We denote the eigenvalues of $[b^i_j]$ by $\widetilde{\k}$. Then $\widetilde{\k_i}\le|\widetilde{\k_i}|\le\phi'+|\g_{ii}|\le\phi'+\eps C_{13}$. Meanwhile, $\widetilde{\k_i}=\phi'-\g_{ii}\ge\phi'-\eps C_{13}.$
	For $\forall i$, $\phi'-\eps C_{13}\le\widetilde{\k_i}\le\phi'+\eps C_{13}$ means
	\begin{equation}\label{F}
		\phi'-\eps C_{13}\le F\le\phi'+\eps C_{13},
	\end{equation}
	where we mention that $F(1,\cdots,1)=1$.
	
	At $P$, $\widetilde{g}^{ij}$ is diagonal with
	$$\widetilde{g}^{11}=\frac{1}{\om^2},\qquad \widetilde{g}^{ii}=1 \text{ for }i\neq1.$$
	For simplicity, we denote by $F=F(\phi'\delta^i_j-\widetilde{g}^{ik}\gamma_{kj})=F(b^i_j)$ and $F^j_i$ is the derivative of $F$ with respect to its argument. We also denote by $G^{ij}:=\frac{\p G}{\p\g_{ij}}, G^{\g_l}:=\frac{\p G}{\p\g_{l}}, G^\g:=\frac{\p G}{\p\g}$.
	Let $\cL'=\p_t-G^{ij}D_iD_j$ be the parabolic operator. Using the Ricci identities $D_l\g_{ij}=D_j\g_{li}+\e_{il}\g_j-\e_{ij}\g_l$ on $\mS^n$, we get
	\begin{equation}\label{5.2}
		\begin{split}
			\cL'O=&-G^{ij}\g_{ik}\g^k_j-G^{ij}e_{ij}|D\g|^2+G^{ij}\g_{i}\g_j\\
			&+G^{\g_l}(|D\g|^2)_l+G^\g|D\g|^2\\
			=&-G^{ij}\g_{ik}\g^k_j-G^{ij}e_{ij}|D\g|^2+G^{ij}\g_{i}\g_j+G^\g|D\g|^2
		\end{split}
	\end{equation}
	by the critical point condition. We can also calculate
	\begin{equation}\label{4.13}
		\begin{split}
			-G^{ij}\g_{ik}\g^k_j=&-\dfrac{1}{\phi \om}F_l^i\widetilde{g}^{lj}\g_{ik}\g^k_j=-\dfrac{1}{\phi \om}\sum_{k\ge2}F^{kk}\g_{kk}^2,\\
			-G^{ij}e_{ij}|D\g|^2+G^{ij}\g_{i}\g_j=&-\dfrac{1}{\phi \om}F_l^i\widetilde{g}^{lj}(e_{ij}|D\g|^2-\g_{i}\g_j)=-\dfrac{1}{\phi \om}\sum_{k\ge2}F^{kk}|D\g|^2,\\
			G^\g|D\g|^2=&\(-\dfrac{\om}{\phi }+\dfrac{\phi}{\om}F_i^i+\dfrac{F\phi'}{\phi\om}\)|D\g|^2.
		\end{split}
	\end{equation}
	We have used notations $F^i_i=\sum_{i\ge1}F^i_i$.
	Substituting (\ref{4.13}) into (\ref{5.2}), we get
	\begin{equation}\label{4.14}
		\begin{split}
			\frac{\cL'O}{|D\g|^2}=&-\dfrac{1}{\phi \om|D\g|^2}\sum_{k\ge2}F^{kk}\g_{kk}^2-\dfrac{1}{\phi \om}\sum_{k\ge2}F^{kk}-\dfrac{\om}{\phi }+\dfrac{\phi}{\om}F_i^i+\dfrac{F\phi'}{\phi\om}.
		\end{split}
	\end{equation}
	We recall that $\phi'-\eps C_{13}\le F\le\phi'+\eps C_{13}$ and $\widetilde{\k_i}\le\phi'+\eps C_{13}.$ Thus
	$$F_i^i=\frac{p_1}{p_2^\frac{1}{2}}=\frac{p_1}{F}\le\frac{\phi'+\eps C_{13}}{\phi'-\eps C_{13}}.$$
	By the $C^2$ estimates, we have $\sum_{k\ge2}F^{kk}\ge C_{14}$. Then
	\begin{equation}\label{4.15}
		\begin{split}
			&-\dfrac{1}{\phi \om}\sum_{k\ge2}F^{kk}-\dfrac{\om}{\phi }+\dfrac{\phi}{\om}F_i^i+\dfrac{F\phi'}{\phi\om}\\
			\le&-\dfrac{C_{14}}{\phi\om}-\dfrac{\om}{\phi}+\dfrac{\phi(\phi'+\eps C_{13})}{\om(\phi'-\eps C_{13})}+\dfrac{\phi'(\phi'+\eps C_{13})}{\phi\om}\\
			=&\dfrac{-C_{13}\phi'-\phi'|D\g|^2+\eps\(C_{13}C_{14}+C_{13}\om^2+C_{13}\phi^2-\eps C_{13}^2\phi'\)}{\phi\om(\phi'-\eps C_{13})}.
		\end{split}
	\end{equation}
	Now if we plug (\ref{4.15}) into (\ref{4.14}) and let $\eps$ be a small enough constant, we have
	\begin{equation}
		\cL'O\le-C_{15}O,
	\end{equation}
	where $C_{15}>0$ is a constant independent of time $t$. The standard maximum principle implies the exponential convergence.
\end{proof}
Then we can complete the proof of Theorem \ref{t4.1}.

\begin{proof of theorem 4.1}
	By Lemma \ref{l4.11}, we have that $\vert D\rho\vert\to0$ exponentially as $t\to\infty$. Hence by the interpolation and the a priori estimates, we can get that $\rho$ converges exponentially to a constant in the $C^\infty$ topology as $t\to\infty$.
\end{proof of theorem 4.1}

\section{Some families of geometric inequalities}
We first prove some monotone quantities. According to these monotonic quantities, we can directly derive some families of geometric inequalities in sphere. The following variational formula has been derived in \cite{WZ}.
\begin{lemma}\cite{WZ}
	Let $M_t$ be a smooth family of closed hypersurfaces in the space form satisfying $\p_tX=\sF\nu$. For all integers $k=1,\cdots,n$ we have
	\begin{equation}\label{x5.1}
		\frac{d}{dt}\(\int_{M_t}\Phi p_kd\mu+kW_{k-1}(\Om_t)\)=\int_{M_t}\((k+1)up_k+(n-k)\Phi p_{k+1}\)\sF d\mu_t.
	\end{equation}
\end{lemma}
To get more comprehensive results, we also need to use the following convergence result, which has been proved in \cite{GL}. The convexity preserved has been proved in \cite{CS,BGL}.
\begin{theorem}\cite{GL,CS,BGL}
Let $M_0$ be a smooth compact, star-shaped hypersurface with
respect to the origin in $\mS^{n+1}$. Then the
evolution equation (\ref{x1.7}) with $F=p_1$ has a smooth solution for all time and converges exponentially to a geodesic slice in the $C^\infty$-topology. If $M_0$ is strictly convex, $M_t$ can preserve convexity.
\end{theorem}
Some monotone quantities along flow (\ref{x1.7}) can be found.
\begin{lemma}\label{l5.6}
	Suppose $M_t$ is a smooth, closed, star-shaped solution to the flow (\ref{x1.7}) with $F=p_1$ in sphere. Then the following hold:
	
	$(\rmnum1)$ $W_0$ is invariant along this flow;
	
	$(\rmnum2)$ $W_{-1}^{\phi'}$ is monotone increasing and $W_{-1}^{\phi'}$ is a constant function if and only if $M_t$ is a slice for each $t$.

	Suppose $M_t$ is a smooth, closed, strictly convex solution to the flow (\ref{x1.7}) with $F=p_1$ in sphere. Then the following hold:
	
	$(\rmnum3)$ $W_m$ is monotone decreasing with $1\le m\le n-1$ and $W_m$ is a constant function if and only if $M_t$ is a slice for each $t$;
	
	$(\rmnum4)$ $\int_{M_t}\Phi p_kd\mu_t+kW_{k-1}(\Om_t)$ is monotone decreasing with $0\le k\le n-2$ and $\int_{M_t}\Phi p_kd\mu+kW_{k-1}(\Om_t)$ is a constant function if and only if $M_t$ is a slice for each $t$.
\end{lemma}
\begin{proof}
	$(\rmnum{1})$: We can directly calculate
	\begin{align*}
		\frac{d}{dt}W_0=\int_{M_t}(\phi'-up_1)d\mu_t=0,
	\end{align*}
	where we used the Minkowski formula. Thus
	$W_0$ is invariant.
	
$(\rmnum{2})$: Recall (\ref{1.6}).
\begin{align*}
	\frac{1}{n+1}\frac{d}{dt}W_{-1}^{\phi'}=&\int_{M_t}\phi'\(\phi'-up_1\)d\mu_t\\
	=&\frac{1}{n}\int_{M_t}\phi'\Delta\Phi d\mu_t=\frac{1}{n}\int_{M_t}|\nn\Phi|^2 d\mu_t\ge0,
\end{align*}
where we have used  $\nn\phi'=-\nn\Phi$ in sphere. Thus
$W_{-1}^{\phi'}$ is a constant function if and only if  $W_{-1}^{\phi'}$ is a slice.
	
	$(\rmnum{3})$: For $1\le m\le n-1$,
	\begin{align*}
		\frac{n+1}{n+1-m}\frac{d}{dt}W_m=\int_{M_t}\left(\phi'p_m-up_1p_m\right)d\mu_t\le\int_{M_t}(\phi'p_m-up_{m+1})d\mu_t=0,
	\end{align*}
	where we used the Newton-MacLaurin inequalities. Thus
	$W_m$ is a constant function if and only if  $M_t$ is a slice.

	$(\rmnum{4})$: For $0\le k\le n-2$, by (\ref{x5.1}),
	\begin{align*}
		&\frac{d}{dt}\(\int_{M_t}\Phi p_kd\mu_t+kW_{k-1}(\Om_t)\)\\
		=&\int_{M_t}\((k+1)up_k+(n-k)\Phi p_{k+1}\)\(\phi'-up_1\) d\mu_t\\
		=&(k+1)\int_{M_t}u\(\phi'p_k-up_1p_k\)d\mu_t+(n-k)\int_{M_t}\Phi\(\phi'p_{k+1}-up_1p_{k+1}\)d\mu_t\\
		\le&(k+1)\int_{M_t}u\(\phi'p_k-up_{k+1}\)d\mu_t+(n-k)\int_{M_t}\Phi\(\phi'p_{k+1}-up_{k+2}\)d\mu_t\\
		=&\int_{M_t}up_{k+1}^{ij}\nn_{ij}^2\Phi d\mu_t+\frac{n-k}{k+2}\int_{M_t}\Phi p_{k+2}^{ij}\nn_{ij}^2\Phi d\mu_t\\
		=&-\(\int_{M_t}p_{k+1}^{ii}\k_i(\nn_{i}\Phi)^2 d\mu_t+\frac{n-k}{k+2}\int_{M_t} p_{k+2}^{ii}(\nn_{i}\Phi)^2d\mu_t\)\\
		\le&0,
	\end{align*}
	where we have used $M_t$ is strictly convex, Newton-MacLaurin inequalities, $\nn_i u=\k_i\nn_i\Phi$ and $\nn^2_{ij}\Phi=\phi'g_{ij}-uh_{ij}$. Thus
	$\int_{M_t}\Phi p_kd\mu+kW_{k-1}(\Om_t)$ is a constant function if and only if  $\int_{M_t}\Phi p_kd\mu+kW_{k-1}(\Om_t)$ is a slice.
\end{proof}
However, due to the form of curvature flow (\ref{x1.7}), the cases $m=n$ in $(\rmnum{3})$ and $k=n-1,n$ in $(\rmnum{4})$ is missed in Lemma \ref{l5.6}. In detail, the Newton-MacLaurin inequalities will fail at these cases. To fill this gap, some monotone quantities can be derived along the inverse curvature flow 
\begin{equation}\label{x1.4}
\frac{\p}{\p t}X=\(\frac{1}{F}-\frac{u}{\phi'}\)\nu.
\end{equation}
The following convergence result has been proved in \cite{SX}.
\begin{theorem}\cite{SX}
	Let $X_0(M)$ be the embedding of a closed $n$-dimensional manifold $M$
	in $\mS^{n+1}$ such that $X_0(M)$ is strictly convex.  Then any solution $X_t$ of (\ref{x1.4}) with $F=\frac{p_k}{p_{k-1}}$ exists for all positive times and converges
	to a geodesic slice in the $C^\infty$-topology.
\end{theorem}
Some monotone quantities along the inverse curvature flow (\ref{x1.4}) can be found.
\begin{lemma}\label{l5.7}
	Suppose $M_t$ is a smooth, closed, strictly convex solution to the flow (\ref{x1.4}) with $F=\frac{p_k}{p_{k-1}}$ in sphere. Then the following hold:
	
	$(\rmnum1)$ $W_{-1}^{\phi'}$ is monotone increasing and $W_{-1}^{\phi'}$ is a constant function if and only if $M_t$ is a slice for each $t$;
	
	$(\rmnum2)$ $W_m$ is monotone decreasing with $1\le k\le m\le n$ and $W_m$ is a constant function if and only if $M_t$ is a slice for each $t$;
	
	$(\rmnum3)$ $\int_{M_t}\Phi p_md\mu_t+mW_{m-1}(\Om_t)$ is monotone decreasing with $1\le k\le m\le n$ and $\int_{M_t}\Phi p_md\mu+mW_{m-1}(\Om_t)$ is a constant function if and only if $M_t$ is a slice for each $t$.
\end{lemma}
\begin{proof}
	(\rmnum{1}): Similar to Lemma \ref{l5.6},
	\begin{align*}
		\frac{1}{n+1}\frac{d}{dt}W_{-1}^{\phi'}=&\int_{M_t}\phi'\(\frac{p_{k-1}}{p_k}-\frac{u}{\phi'}\)d\mu_t\\
		\ge&\int_{M_t}(\frac{\phi'}{p_1}-u)d\mu_t\ge0,
	\end{align*}
	where we have used the Newton-MacLaurin inequalities and  Brendle's Heintze-Karcher-type inequality \cite{BS}. Thus
	$W_{-1}^{\phi'}$ is a constant function if and only if  $W_{-1}^{\phi'}$ is a slice.
	
	(\rmnum{2}): For $k\le m\le n$,
	\begin{align*}
		\frac{n+1}{n+1-m}\frac{d}{dt}W_m=&\int_{M_t}\frac{1}{\phi'}\(\phi'\frac{p_{k-1}p_m}{p_k}-up_m\)d\mu_t\\
		\le&\int_{M_t}\frac{1}{\phi'}\(\phi'p_{m-1}-up_m\)d\mu_t\\
		=&\frac{1}{m}\int_{M_t}\frac{p_m^{ii}}{\phi'}\nn_{ii}^2\Phi d\mu_t\\
		=&\frac{1}{m}\int_{M_t}-\frac{p_m^{ii}}{\phi'^2}|\nn_i\Phi|^2d\mu_t\le0,
	\end{align*}
	where we used the Newton-MacLaurin inequalities and $\nn\phi'=-\nn\Phi'$ in sphere. Thus
	$W_m$ is a constant function if and only if  $M_t$ is a slice.
	
	(\rmnum{3}): By (\ref{x5.1}),
	\begin{align*}
		&\frac{d}{dt}\(\int_{M_t}\Phi p_md\mu_t+mW_{m-1}(\Om_t)\)\\
		=&\int_{M_t}\((m+1)up_m+(n-m)\Phi p_{m+1}\)\(\frac{p_{k-1}}{p_k}-\frac{u}{\phi'}\) d\mu_t\\
		=&(m+1)\int_{M_t}\frac{u}{\phi'}\(\phi'\frac{p_{k-1}p_m}{p_k}-up_m\)d\mu_t+(n-m)\int_{M_t}\frac{\Phi}{\phi'}\(\phi'\frac{p_{k-1}p_{m+1}}{p_k}-up_{m+1}\)d\mu_t\\
		\le&(m+1)\int_{M_t}\frac{u}{\phi'}\(\phi'p_{m-1}-up_{m}\)d\mu_t+(n-m)\int_{M_t}\frac{\Phi}{\phi'}\(\phi'p_{m}-up_{m+1}\)d\mu_t\\
		=&\frac{m+1}{m}\int_{M_t}\frac{u}{\phi'}p_{m}^{ij}\nn_{ij}^2\Phi d\mu_t+\frac{n-m}{m+1}\int_{M_t}\frac{\Phi}{\phi'} p_{m+1}^{ij}\nn_{ij}^2\Phi d\mu_t\\
		=&-\(\frac{m+1}{m}\int_{M_t}p_{m}^{ii}\frac{\phi'\k_i+u}{\phi'^2}(\nn_{i}\Phi)^2 d\mu_t+\frac{n-m}{m+1}\int_{M_t} \frac{p_{m+1}^{ii}}{\phi'^2}(\nn_{i}\Phi)^2d\mu_t\)\\
		\le&0,
	\end{align*}
	where we have used $M_t$ is strictly convex, Newton-MacLaurin inequalities, $\nn_i u=\k_i\nn_i\Phi$, $\nn^2_{ij}\Phi=\phi'g_{ij}-uh_{ij}$, $\nn\phi'=-\nn\Phi$ and $\phi'+\Phi=1$ in sphere. Thus
	$\int_{M_t}\Phi p_kd\mu+kW_{k-1}(\Om_t)$ is a constant function if and only if  $\int_{M_t}\Phi p_kd\mu+kW_{k-1}(\Om_t)$ is a slice.
\end{proof}

By Lemmas \ref{l5.6} and \ref{l5.7}, we can derive Theorems \ref{t5.7} and \ref{t5.9} directly. We give the proof of  (\ref{5.9}) as an example.

\begin{proof of (1.13)}
	Denote the limit slice by $S_{r_\infty}=\p\Om_\infty$. Then
	\begin{align*}
		\int_{\p\Om}\Phi p_kd\mu+kW_{k-1}(\Om)\ge&\lim_{t\rightarrow\infty}\(\int_{\p\Om_t}\Phi p_kd\mu+kW_{k-1}(\Om_t)\)\\
		=&\int_{\p\Om_\infty}\Phi p_kd\mu+kW_{k-1}(\Om_\infty)=(\xi_k+kf_{k-1})(r_\infty),\\
		W_{-1}^{\phi'}(\Om)\le&\lim_{t\rightarrow\infty}W_{-1}^{\phi'}(\Om_t)=W_{-1}^{\phi'}(\Om_\infty)=h_{-1}(r_\infty).
	\end{align*}
	Since functions $\xi_k+kf_{k-1}$ and $h_l$ are strictly increasing, inequality (\ref{5.9}) holds. The rigidity of (\ref{5.9}) follows from the rigidity of the monotonicity. This completes this proof.
\end{proof of (1.13)}

Finally, we shall raise three questions that are well worth considered.
\begin{conjecture}
	Let $\Om$ be a strictly convex domain with smooth boundary $M$ in $\mS^{n+1}$. If $0\le k\le n$, $0\le l\le k$, there holds
\begin{equation}
	\int_{M}\Phi p_kd\mu+kW_{k-1}(\Om)\ge (\xi_k+kf_{k-1})\circ f_{l}^{-1}(W_{l}(\Om)).
\end{equation}
Equality holds if and only if $\Om$ is a geodesic ball centered at the origin.
\end{conjecture}
\begin{conjecture}
	Let $\Om$ be a strictly convex domain with smooth boundary $M$ in $\mS^{n+1}$. If $0\le k\le n$, $0\le l\le k$, there holds
	\begin{equation}
		\int_{M}\Phi p_kd\mu+kW_{k-1}(\Om)\ge (\xi_k+kf_{k-1})\circ h_{l}^{-1}(W_{l}^{\phi'}(\Om)).
	\end{equation}
	Equality holds if and only if $\Om$ is a geodesic ball centered at the origin.
\end{conjecture}
The third question is as follows: How to derive the $C^2$ estimates of flow (\ref{x1.7}) with $F=p_k^\frac{1}{k}$ for strictly $k$-convex initial hypersurface in sphere? If we solve this problem and derive the lower bound of $F$, we can derive (\ref{5.15}) for strictly $m$-convex domains.

\section{Reference}


\end{document}